\documentclass[final,a4paper]{amsart}

\usepackage[english]{babel}
\usepackage{amsmath,amsfonts,amssymb,amsthm}
\usepackage[utopia]{mathdesign}
\usepackage{graphicx}
\usepackage{mathrsfs}
\usepackage{enumerate} 
\usepackage{array}
\usepackage{showkeys}

 \usepackage{hyperref}
 \hypersetup{colorlinks=true,linkcolor=blue,citecolor=blue,linktoc=page}
 
\newcommand{\eq}{\begin{equation}}
\newcommand{\qe}{\end{equation}}

\renewcommand{\S}{\mathcal{S}}

\newcommand{\R}{\mathbb{R}}                     
\newcommand{\abs}[1]{\left\vert#1\right\vert}
\newcommand{\norm}[1]{\big\lVert#1\big\lVert}
\newcommand{\X}{\mathbf{X}}   

\newtheoremstyle{theo}
{}
{}
{\itshape}
{\parindent}
{\bf}
{\ ---}
{.5em}
{}%
\theoremstyle{theo}
\newtheorem{theorem}{Theorem}[section]
\newtheorem{lemma}[theorem]{Lemma}
\newtheorem{proposition}[theorem]{Proposition}

\newtheoremstyle{def}%
{}
{}
{\itshape}
{\parindent}
{\bf}
{\ ---}
{.5em}
{}%
\theoremstyle{def}
\newtheorem{definition}{Definition}

\theoremstyle{remark}

\begin{document}
\sloppy

\pagestyle{headings} 
\title{Optimal designs for Lasso and Dantzig selector using Expander Codes}
\date{\today}
\keywords{Lasso; Dantzig Selector; Expander; Restricted Eigenvalue;}

\author{Yohann de Castro}
\address{YdC is with the D\'epartement de Math\'ematiques (CNRS UMR 8628), B\^atiment 425, Facult\'e des Sciences d'Orsay, Universit\'e Paris-Sud 11, F-91405 Orsay Cedex, France.}
\email{yohann.decastro@math.u-psud.fr}

\begin{abstract}
We investigate the high-dimensional regression problem using adjacency matrices of unbalanced expander graphs. In this frame, we prove that the $\ell_{2}$-prediction error and the $\ell_{1}$-risk of the lasso and the Dantzig selector are optimal up to an explicit multiplicative constant. Thus we can estimate a high-dimensional target vector with an error term similar to the one obtained in a situation where one knows the support of the largest coordinates in advance. 

Moreover, we show that these design matrices have an explicit restricted eigenvalue. Precisely, they satisfy the restricted eigenvalue assumption and the compatibility condition with an explicit constant. 

Eventually, we capitalize on the recent construction of unbalanced expander graphs due to Guruswami, Umans, and Vadhan, to provide a deterministic polynomial time construction of these design matrices.
\end{abstract}

\maketitle
\section{Introduction}

One of the recent breakthrough in Statistics and Image processing has been brought by the idea that one can recover a high-dimensional target vector from few (non-adaptative) linear measurements as soon as the target vector depends only on few coefficients in a well-chosen basis. Some seminal works can be found in \cite{MR1639094,MR1379242,MR2300700,MR2382644} and references therein. One of the major observation in this field relies on the fact that an optimal way of gathering information on the target is at random. Recently much emphasis has been put on the lasso \eqref{Lasso}, the Dantzig selector \eqref{Dantzig} and their ability to recover an almost sparse vector $\beta^\star\in\R^p$ from $n$ linear measurements in the high-dimensional setting, i.e., where $n$ is much smaller than $p$. Consider the linear model:
\begin{equation*}
 y =\X\beta^\star+z\,,
\end{equation*}
where $\X\in\R^{n\times p}$ is a design matrix and $z\in\R^n$ a random noise vector.  Assume that $z=(z_i)_{i=1}^n$ is such that $z_{1},\dots,z_{n}$ are identically distributed according to a centered Gaussian random variable with variance $\sigma^2$. At first glance, it seems impossible to recover $\beta^{\star}$ from the noisy observation $y$. As a matter of fact, we may ask whether there exists an integer $s$, as large as possible, such that one can recover the $s$ largest, in absolute value, coordinates of $\beta^{\star}$, i.e., a sparse approximation. This issue can be addressed by considering the lasso \cite{MR1379242}:
\begin{equation}\label{Lasso}
 \beta^{{\ell}} \in\arg \min_{\beta\in\R^p}\big\{
\norm{y-\X\beta}_2^2+\lambda_{{\ell}}\norm\beta_1\big\},
\end{equation}
where $\lambda_{{\ell}}$ is a tuning parameter; or a solution to the $\ell_1$-regularized problem called Dantzig selector \cite{MR2382644}:
\begin{equation}\label{Dantzig}
 \beta^{\mathrm{d}}\in\arg \min_{\beta\in\R^p}
\norm{\beta}_1\ \ \mathrm{s.t.}\ \ \lVert{\X^\top(y-\X\beta)}\lVert_\infty\leq\lambda_{\mathrm d}\,,
\end{equation}
where $\lambda_{\mathrm d}$ a tuning parameter. A large literature has shown that these estimators are suited to uncover sparse approximations from few observations as soon as the design matrix $\X$ satisfies, for instance, the $\mathrm{RIP}_2$ property \cite{MR2230846} or the coherence property \cite{MR2543688}. Although one can prove that random design matrices satisfy these properties with high-probability, providing a deterministic design matrix is challenging. However, a recent breakthrough \cite{bourgain2011explicit} gives an explicit construction of design matrices satisfying the $\mathrm{RIP}_2$ property of order $n^{1/2+\eta}$, for some small $\eta>0$. 

In this article, we consider adjacency matrix of an unbalanced expander graph with expansion constant no greater than $1/12$, see Definition \ref{Def_Expander}. For the sake of simplicity, we refer to it as expander design matrices. Observe that the expander design matrices satisfy neither the $\mathrm{RIP}_2$ property nor the coherence property, see Section \ref{Assumption}. Hence we have to carry out a new analysis of the lasso, and the Dantzig selector, in the expander design matrix framework. This analysis is inspired by a companion paper of the author which studies these estimators in the geometric functional analysis frame. More precisely, the article \cite{de2013remark} analyzes the design matrices with small distortion, i.e., for which the kernel has small $\ell_{2}$-intersection with the $\ell_{1}$-ball. Nevertheless, we follow a quite different path here since we are not able to compute the distortion of the kernel of a expander design matrix. Interestingly, our proof show that the expander design matrices satisfy four standard properties in the high-dimensional regression theory: the restricted eigenvalue assumption \cite{MR2533469}, the compatibility condition \cite{MR2576316}, the $\mathbf{H}_{s,1}(1/5)$ condition \cite{JN10}, and the universal distortion property \cite{de2013remark}, see Section \ref{Assumption}. Furthermore, we prove that these conditions hold with explicit constants, see Proposition \ref{prop:REExpander}.  Moreover, it gives an example of design matrix that satisfies these four properties but satisfies neither the $\mathrm{RIP}_2$ property nor the coherence property.

\begin{theorem}
\label{thm:thmIntro1}
Let $\Phi\in\{0,1\}^{n\times p}$ be the adjacency matrix of a $(2s,\varepsilon)$-unbalanced
expander with an expansion  constant $\varepsilon\leq1/12$ and left degree $d$. Set $\X=(1/{\sqrt d})\Phi$. If the quantities $1/\varepsilon$, $d$ are smaller than $p$ and $\lambda_{{\ell}}=20\,\sigma\sqrt{\log p}$, then it holds:
\begin{equation}\label{eq:RiskLassoExpander}
\norm{\beta^{{\ell}}-\beta^{\star}}_{1}\leq4
\min_{\substack{\S\subseteq\{1,\dotsc,p\},\\ \abs{\S}\leq s.}}
\Big[1.5\,\lambda_{{\ell}}\,\abs{\S}+\norm{\beta^\star_{\S^c}}_{1}\Big],
\end{equation}
where $\beta_{\S^c}$ denotes the sub-vector of $\beta$ obtained by removing all
the coordinates having indexes in $\S$, and:
\begin{equation}\label{eq:PredictionLassoExpander}
\norm{\X\beta^{{\ell}}-\X\beta^{\star}}_{{2}}\leq\min_{\substack{\S\subseteq\{1,\dotsc,p\},\\ \abs{\S}\leq s.}}
\Bigg[4.8\,\lambda_{{\ell}}\,\sqrt{\abs{\S}}+0.9\,\frac{\norm{\beta^\star_{\S^c}}_{1}}{\sqrt{\abs{\S}}}\Bigg],
\end{equation}
with a probability greater than $1-1/({p\sqrt{2\pi\log p}})$.
\end{theorem}
\noindent
These inequalities are optimal up to a multiplicative constant less than $6$. Indeed, observe that Equation \eqref{eq:RiskLassoExpander} shows that the $\ell_{1}$-risk (the left hand side) is upper bounded by the optimal risk (given by the best approximation by $s$-sparse vectors $\norm{\beta^\star_{\S^c}}_{1}$) and a soft-thresholding term $\lambda_{{\ell}}\,s$, see Figure \ref{Fig:fig1}. It shows that we recover the $s$ largest coefficients of $\beta^{\star}$ up to a coordinate wise error of the order $\lambda_{{\ell}}$. Note this is essentially the noise level $\sigma$ up to the square root of a log factor. 

\begin{figure}[!h]
\center
\includegraphics[height=5cm]{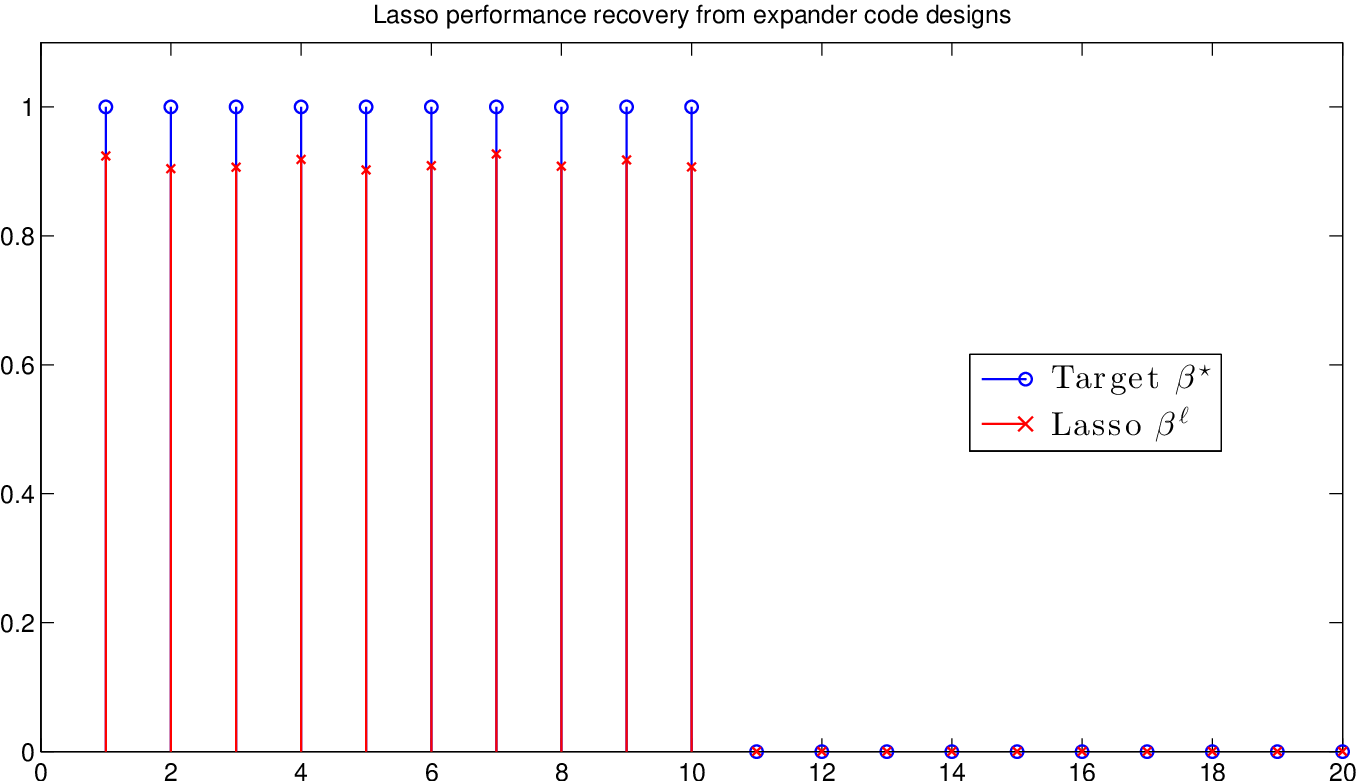}
\caption{Lasso estimator $\beta^{\ell}$ obtained from expander code designs exhibits a soft-thresholding phenomenon: coordinate wise error is of the order $\lambda_{{\ell}}$. This well-established behavior is captured in the upper bound \eqref{eq:RiskLassoExpander}.}\label{Fig:fig1}
\end{figure}

Moreover, the $\ell_{2}$-prediction error \eqref{eq:PredictionLassoExpander} is upper bounded by $\sigma\sqrt s$ (up to the square root of a log factor) which is the best prediction error in the linear model. As a matter of fact, it corresponds to a situation where one performs ordinary least squares on the set of columns of $\X$ given by the support of the $s$ largest entires of $\beta^{\star}$. The remaining term ${\norm{\beta^\star_{\S^c}}_{1}}/{\sqrt s}$ can be understood has the approximation error by $s$-sparse vectors. 

Eventually, observe the upper bound given by Theorem \ref{thm:thmIntro1} is optimal in the sense that it is the $\ell_{1}$-risk and the $\ell_{2}$-prediction error that we get if we would have known the support of the $s$ largest coordinates, in absolute value, of the signal $\beta^{\star}$ in advance and the remaining term $\beta^\star_{\S^c}$. We prove the same result for the Dantzig selector.
\begin{theorem}
Let $\Phi\in\{0,1\}^{n\times p}$ be the adjacency matrix of a $(2s,\varepsilon)$-unbalanced
expander with an expansion  constant $\varepsilon\leq1/12$ and left degree $d$. Set $\X=(1/{\sqrt d})\Phi$. If $1/\varepsilon$, $d$ are smaller than $p$ and $\lambda_{\mathrm d}=9\,\sigma\sqrt{\log p}$, then it holds:
\begin{equation*}
\norm{\beta^{\mathrm d}-\beta^{\star}}_{1}\leq9\min_{\substack{\S\subseteq\{1,\dotsc,p\},\\ \abs{\S}\leq s.}}
\Big[1.5\,\lambda_{\mathrm d}\,\abs{\S}+\norm{\beta^\star_{\S^c}}_{1}\Big],
\end{equation*}
\begin{equation*}
\norm{\X\beta^{\mathrm d}-\X\beta^{\star}}_{{2}}\leq\min_{\substack{\S\subseteq\{1,\dotsc,p\},\\ \abs{\S}=\leq s.}}
\Bigg[4.8\,\lambda_{\mathrm d}\,\sqrt{\abs{\S}}+0.9\,\frac{\norm{\beta^\star_{\S^c}}_{1}}{\sqrt{\abs{\S}}}\Bigg],
\end{equation*}
with a probability greater than $1-1/({p\sqrt{2\pi\log p}})$.
\end{theorem}

In Section \ref{Deterministic Design}, we prove that the expander design matrices can be deterministically constructed \cite{MR2590822} based on Paravaresh-Vardy codes \cite{Par}. In this frame, our analysis gives the following results.

\begin{proposition}\label{prop:principalLassoDeterministeIntro}
There exists a positive universal constant $\theta_{0}$ such that the following holds. For all $s,p$ such that $8\leq s\leq p/2$, there exists a deterministic polynomial time construction of a design matrix $\X\in\R^{n\times p}$ where:
\begin{equation}\label{eq:bornmesuredeterministes}
n\leq e\,s\,(12\,\theta_{0}\,\log p\log s )^{3\log(s)},
\end{equation}
with $e$ the Euler number, such that for $\lambda_{{\ell}}=20\,\sigma\sqrt{\log n}$, we have:
\begin{equation*}
\norm{\beta^{{\ell}}-\beta^{\star}}_{1}\leq4
\min_{\substack{\S\subseteq\{1,\dotsc,p\},\\ \abs{\S}\leq s.}}
\Big[1.5\,\lambda_{{\ell}}\,\abs{\S}+\norm{\beta^\star_{\S^c}}_{1}\Big],
\end{equation*}
\begin{equation*}
\norm{\X\beta^{{\ell}}-\X\beta^{\star}}_{{2}}\leq\min_{\substack{\S\subseteq\{1,\dotsc,p\},\\ \abs{\S}\leq s.}}
\Bigg[4.8\,\lambda_{{\ell}}\,\sqrt{\abs{\S}}+0.9\,\frac{\norm{\beta^\star_{\S^c}}_{1}}{\sqrt{\abs{\S}}}\Bigg],
\end{equation*}
with a probability greater than $1-1/({p\sqrt{2\pi\log p}})$.
\end{proposition}

\noindent
Similarly, we have the following result for the Dantzig selector.

\begin{proposition}\label{prop:principalDantzigDeterministeIntro}
There exists a positive universal constant $\theta_{0}$ such that the following holds. For all $s,p$ such that $8\leq s\leq p/2$, there exists a deterministic polynomial time construction of a design matrix $\X\in\R^{n\times p}$ where:
\[
n\leq e\,s\,(12\,\theta_{0}\,\log p\log s )^{3\log(s)},
\]
such that for $\lambda_{\mathrm d}=9\,\sigma\sqrt{\log n}$, we have:
\begin{equation*}
\norm{\beta^{\mathrm d}-\beta^{\star}}_{1}\leq9
\min_{\substack{\S\subseteq\{1,\dotsc,p\},\\ \abs{\S}\leq s.}}
\Big[1.5\,\lambda_{\mathrm d}\,\abs{\S}+\norm{\beta^\star_{\S^c}}_{1}\Big],
\end{equation*}
\begin{equation*}
\norm{\X\beta^{\mathrm d}-\X\beta^{\star}}_{{2}}\leq\min_{\substack{\S\subseteq\{1,\dotsc,p\},\\ \abs{\S}\leq s.}}
\Bigg[4.8\,\lambda_{\mathrm d}\,\sqrt{\abs{\S}}+0.9\,\frac{\norm{\beta^\star_{\S^c}}_{1}}{\sqrt{\abs{\S}}}\Bigg],
\end{equation*}
with a probability greater than $1-1/({p\sqrt{2\pi\log p}})$.
\end{proposition}

\noindent
It shows that, with a sensible budget of observations \eqref{eq:bornmesuredeterministes}, one can construct in polynomial time a design matrix $\X$ for which the lasso, or the Dantzig selector, recovers a sparse approximation of all signal $\beta^{\star}$.

This articles folds into four parts. First, we review some results on unbalanced expander graphs. Then we give the proofs of the main results. The third part gives evidence that expander design matrices matrices satisfy neither the $\mathrm{RIP}_2$ property nor the coherence property but they satisfy the restricted eigenvalue assumption, the compatibility condition, the $\mathbf{H}_{s,1}(1/5)$ condition, and the universal distortion property. Moreover, we give an explicit formulation of the constants involved in these four properties. Last but not least, we present a deterministic polynomial time construction of expander design matrices in the last part.

\section{Unbalanced expander graphs}\label{Expander}
We recall some standard facts about adjacency matrices of unbalanced expander graphs. Let us denote a {bipartite} graph $G=(A,B,E)$ where the set of the left vertices is denoted $A$ and has size $p$, the set of the right vertices is denoted $B$  and has size $n$, and $E$ is the set of the edges between $A$ and $B$. Suppose that $G$ has regular left degree $d$, i.e., every vertex in $A$ has exactly $d$ neighbors in $B$, then consider the {normalized adjacency matrix} $\X\in\R^{n\times p}$ given by:
\begin{equation*}
 \X_{ij}=\begin{cases}
          \frac1{\sqrt d} & \mathrm{if}\ i\ \mathrm{is\ connected\ to}\ j\,,\\
	  0 & \mathrm{otherwise}\,,
         \end{cases}
\end{equation*}
where $i\in \{1,\dotsc,n\}$ and $j\in\{1,\dotsc,p\}$. 

\begin{definition}[$(s,\varepsilon)$-unbalanced expander]\label{Def_Expander}
An $(s,\varepsilon)$-unbalanced expander is a bipartite simple graph $G=(A,B,E)$ with
left degree $d$ such that for any $S\subset A$ with $\abs S\leq s$, the set of
neighbors $N(S)$ of $S$ has size:
\begin{equation}\label{Expansion}
 \abs{N(S)}\geq(1-\varepsilon)\,d\abs S\,.
\end{equation}
The parameter $\varepsilon$ is called the {expansion constant}.
\end{definition}
\noindent Subsequently we shall work
with $\varepsilon= 1/12$. Note $\varepsilon$ is fixed and {does not}
depend on other parameters. One important property of expander design matrices is that they satisfy the $\mathrm{RIP}_{1}$ property.

\begin{lemma}[Theorem 1 in \cite{Ber}]
\label{Lemma UP}
If the quantities $1/\varepsilon$ and $d$ are smaller than $p$ then $(1/\sqrt d)\X$ satisfies the $\mathrm{RIP}_1$ property: 
$\forall\gamma\in\R^p\,,\ \forall S\subseteq\{1,\dotsc,p\}$ such that $\abs S\leq s$,
\begin{equation} \label{eq:RIP1}
{(1-2\varepsilon)}\,\norm{\gamma_S}_1\leq(1/\sqrt d)\norm{\X\gamma_{S}}_1\leq\norm{\gamma_S}_1\,.
\end{equation}
\end{lemma}
\noindent
In Section \ref{Deterministic Design}, we present the state-of-the-art constructions in unbalanced expander graph theory.

\section{Main results and proofs}
\subsection{Lasso case}
\begin{theorem}\label{thm:principalLasso}
Let $\Phi\in\{0,1\}^{n\times p}$ be the adjacency matrix of a $(2s,\varepsilon)$-unbalanced
expander with an expansion  constant $\varepsilon\leq1/12$ and left degree $d$. Set $\X=(1/{\sqrt d})\Phi$. If the quantities $1/\varepsilon$ and $d$ are smaller than $p$ then for all $\lambda_{{\ell}}>10\,\sigma\sqrt{\log p}/3$, it holds:
\begin{equation*}
\norm{\beta^{{\ell}}-\beta^{\star}}_{1}\leq\frac{2}{\big(1-\frac{\lambda_0}{\lambda_{{\ell}}}
\big)-\frac{2}{5}
} \
\min_{\substack{\S\subseteq\{1,\dotsc,p\},\\ \abs{\S}=k,\ k\leq s.}}
\Big[\frac{36}{25}\,\lambda_{{\ell}}\,k+\norm{\beta^\star_{\S^c}}_{1}\Big],
\end{equation*}
where $\lambda_0:=2\,\sigma\sqrt{\log p}$, and:
\begin{equation*}
\norm{\X\beta^{{\ell}}-\X\beta^{\star}}_{{2}}\leq\min_{\substack{\S\subseteq\{1,\dotsc,p\},\\ \abs{\S}=k,\
k\leq s.}}
\Bigg[\frac{24}5\lambda_{{\ell}}\,\sqrt k+\frac{5\norm{\beta^\star_{\S^c}}_{1}}{6\sqrt k}\Bigg],
\end{equation*}
with a probability greater than $1-1/p{\sqrt{2\pi\log p}}$.
\end{theorem}
\subsection{Proof of Theorem \ref{thm:principalLasso}}
\noindent
We denote by $\Phi$ the adjacency matrix so that $\Phi=\sqrt d\X$. We begin with a lemma. 
\begin{lemma}\label{Lem:part1}
For all $\gamma\in\R^{p}$ and for all $S\subseteq\{1,\dotsc,p\}$ such that $\abs S\leq s$,
\begin{equation*}
\norm{\gamma_{S}}_{1}\leq\frac{\sqrt{s}}{1-2\varepsilon}\norm{\X\gamma}_{2}+\frac{2\varepsilon}{1-2\varepsilon}\norm\gamma_{1}\,.
\end{equation*}
\end{lemma}
\begin{proof}
Without loss of generality, assume $S$ consists of the $s$ largest, in magnitude, coefficients of $\gamma$. Partition the coordinates into sets $S_{0}$, $S_{1}$, $S_{2}$, ... ,$S_{q}$, such that the coordinates in the set $S_{l}$ are not larger than the coordinates in $S_{l-1}$, $l\geq1$, and all sets but the last one $S_{q}$ have size $s$. Observe that we can choose $S_{0}=S$. Let $\Phi'$ be a sub matrix of $\Phi$ containing rows from $N(S)$, the set of neighbors of $S$. Using Cauchy-Schwartz inequality, it holds:
\begin{equation*}
\sqrt{sd}\norm{\Phi\gamma}_{2}\geq \sqrt{sd}\norm{\Phi'\gamma}_{2}\geq \frac{\sqrt{sd}}{\sqrt{\lvert{N(S)}\lvert}}\norm{\Phi'\gamma}_{1}\geq \norm{\Phi'\gamma}_{1}.
\end{equation*}
We finish the proof with a standard argument, see Lemma 11 in \cite{Ber}. From \eqref{eq:RIP1} we get that:
\begin{align*}
\norm{\Phi'\gamma}_{1}\geq &\,\norm{\Phi'\gamma_{S}}_{1}-\sum_{l\geq1}\sum_{(i,j)\in E, i\in S_{l},j \in N(S)}\abs{\gamma_{i}}\\
\geq &\, d(1-2\varepsilon)\norm{\gamma_{S}}_{1}
-\sum_{l\geq1}\lvert E\cap(S_{l}\times N(S))\lvert\,\min_{i\in S_{l-{1}}}\abs{\gamma_{i}}\\
\geq &\,  d(1-2\varepsilon)\norm{\gamma_{S}}_{1}
-\frac{1}{s}\sum_{l\geq1}\lvert E\cap(S_{l}\times N(S))\lvert\,\norm{\gamma_{S_{l-{1}}}}_{1}.
\end{align*}
From the expansion property \eqref{Expansion} of $G$ it follows that, for $l\geq1$, we have $\lvert N(S\cup S_{l})\lvert\geq d(1-\varepsilon)\lvert S\cup S_{l}\lvert$. Hence at most $2\varepsilon d s$ edges can cross from $S_{l}$ to $N(S)$, and so
\begin{align*}
\sqrt{sd}\norm{\Phi\gamma}_{2}& \geq d(1-2\varepsilon)\norm{\gamma_{S}}_{1}-2\varepsilon d\sum_{l\geq1}\norm{\gamma_{S_{l-{1}}}}_{1}/s\\
& \geq d(1-2\varepsilon)\norm{\gamma_{S}}_{1}-2\varepsilon d\norm\gamma_{1}.
\end{align*}
The result follows since $\Phi=\sqrt d\X$.
\end{proof}
\noindent
Since $1/(1-2\varepsilon)\leq6/5$ and  $2\varepsilon/(1-2\varepsilon)\leq1/5$, the aforementioned lemma shows that for all $\gamma\in\R^{p}$, for all $S\subseteq\{1,\dotsc,p\}$ such that $\abs S\leq s$,
\begin{equation}\label{eq:UDP18}
5\norm{\gamma_{S}}_{1}\leq{6}\sqrt{s}\norm{\X\gamma}_{2}+\norm\gamma_{1}\,.
\end{equation}
We have the following lemma.
\begin{lemma}\label{lem:remresult}
 Assume that $\X$ satisfies \eqref{eq:UDP18}. Conditioned on the event:
 \begin{equation}\label{eq:BoundNoise}
 \norm{\X^{\top}z}_{\infty}\leq \lambda_0\,,
 \end{equation}
 for all $ \lambda_{{\ell}}>{5\lambda_0}/3$, it holds:
\begin{equation*}
\norm{\beta^{{\ell}}-\beta^{\star}}_{1}\leq\frac{2}{\big(1-\frac{\lambda_0}{\lambda_{{\ell}}}
\big)-\frac{2}{5}
} \
\min_{\substack{\S\subseteq\{1,\dotsc,p\},\\ \abs{\S}=k,\ k\leq s.}}
\Big[\frac{36}{25}\,\lambda_{{\ell}}\,k+\norm{\beta^\star_{\S^c}}_{1}\Big];
\end{equation*}
\begin{equation*}
\norm{\X\beta^{{\ell}}-\X\beta^{\star}}_{{2}}\leq\min_{\substack{\S\subseteq\{1,\dotsc,p\},\\ \abs{\S}=k,\
k\leq s.}}
\Bigg[\frac{24}5\lambda_{{\ell}}\,\sqrt k+\frac{5\norm{\beta^\star_{\S^c}}_{1}}{6\sqrt k}\Bigg].
\end{equation*}
\end{lemma}
\begin{proof}
This lemma is given by Theorem 2.1 and Theorem 2.2 in \cite{de2013remark}.
\end{proof}
\noindent
Lemma \ref{lem:remresult} gives the result on the event \eqref{eq:BoundNoise}. Thus, we need to give an upper bound on the probability of this event. Combining Lemma \ref{lem:remresult} and Proposition \ref{prop:NoiseLevel} (choose $\lambda_{0}:=\lambda_{0}(1)$), we finish the proof.

\subsection{Dantzig selector case}


\begin{theorem}\label{thm:PrincipalDantz}
Let $\Phi\in\{0,1\}^{n\times p}$ be the adjacency matrix of a $(2s,\varepsilon)$-unbalanced
expander with an expansion  constant $\varepsilon\leq1/12$ and left degree $d$. Set $\X=\frac1{\sqrt d}\Phi$. If the quantities $1/\varepsilon$ and $d$ are smaller than $p$ then for all $\lambda_{\mathrm d}>3\,\sigma\sqrt{\log p}$, it holds:
\begin{equation*}
\norm{\beta^{\mathrm d}-\beta^{\star}}_{1}\leq\frac{4}{\big(1-\frac{\lambda_0}{\lambda_{\mathrm d}}
\big)-\frac{1}{3}
} \
\min_{\substack{\S\subseteq\{1,\dotsc,p\},\\ \abs{\S}=k,\ k\leq s.}}
\Big[\frac{36}{25}\,\lambda_{\mathrm d}\,k+\norm{\beta^\star_{\S^c}}_{1}\Big],
\end{equation*}
where $\lambda_0:=2\,\sigma\sqrt{\log p}$, and:
\begin{equation*}
\norm{\X\beta^{\mathrm d}-\X\beta^{\star}}_{{2}}\leq\min_{\substack{\S\subseteq\{1,\dotsc,p\},\\ \abs{\S}=k,\
k\leq s.}}
\Bigg[\frac{24}5\lambda_{\mathrm d}\,\sqrt k+\frac{5\norm{\beta^\star_{\S^c}}_{1}}{6\sqrt k}\Bigg],
\end{equation*}
with a probability greater than $1-1/p{\sqrt{2\pi\log p}}$.
\end{theorem}
\subsection{Proof of Theorem \ref{thm:PrincipalDantz}}
\noindent
Note that \eqref{eq:UDP18} holds. We begin with a lemma.
\begin{lemma}\label{lem:remresultDantzig}
 Assume that $\X$ satisfies \eqref{eq:UDP18}. Conditioned on the event $\norm{\X^{\top}z}_{\infty}\leq \lambda_0$, for all $ \lambda_{{\ell}}>3\,{\lambda_0}/2$, it holds:
\begin{equation*}
\norm{\beta^{\mathrm d}-\beta^{\star}}_{1}\leq\frac{4}{\big(1-\frac{\lambda_0}{\lambda_{\mathrm d}}
\big)-\frac{1}{3}
} \
\min_{\substack{\S\subseteq\{1,\dotsc,p\},\\ \abs{\S}=k,\ k\leq s.}}
\Big[\frac{36}{25}\,\lambda_{\mathrm d}\,k+\norm{\beta^\star_{\S^c}}_{1}\Big],
\end{equation*}
\begin{equation*}
\norm{\X\beta^{\mathrm d}-\X\beta^{\star}}_{{2}}\leq\min_{\substack{\S\subseteq\{1,\dotsc,p\},\\ \abs{\S}=k,\
k\leq s.}}
\Bigg[\frac{24}5\lambda_{\mathrm d}\,\sqrt k+\frac{5\norm{\beta^\star_{\S^c}}_{1}}{6\sqrt k}\Bigg].
\end{equation*}
\end{lemma}
\begin{proof}
This lemma is a consequence of Theorem 2.1 and 2.2 in \cite{de2013remark}.
\end{proof}
\noindent
Combining Lemma \ref{lem:remresultDantzig} and Proposition \ref{prop:NoiseLevel}, we finish the proof.

\subsection{Noise control}

\begin{proposition}
\label{prop:NoiseLevel}
Suppose $z=(z_i)_{i=1}^n$ is a centered Gaussian noise with variance $\sigma^2$ such that the
$z_i$'s are $\mathcal{N}\big(0,{\sigma^2}\big)$-distributed and could be correlated. {Then}, for $t\geq1$ and $ \lambda_{0}(t)=(1+t)\,\sigma\sqrt{\log p}$, it holds:
\begin{equation}\label{High Proba}
 \mathbb
P\big(\norm{\X^\top z}_\infty>\lambda_{0}(t)\big)\leq\frac{\sqrt
2}{(1+t)\sqrt{\pi\log p}\, p^{\frac{(1+t)^2}2-1}}\,.
\end{equation}
\end{proposition}
\begin{proof}
This is a standard result, see Lemma A.1 in \cite{de2013remark} for instance.
\end{proof}

Note that, by replacing $\lambda_{0}$ by $\lambda_{0}(t)$ in the statements of our theorems, it is possible
to replace all the probabilities of the form $1-\eta_n$ by probabilities of the form \eqref{High
Proba}. 

\section{Standard conditions}\label{Assumption}

\subsection{RIP properties}

We begin with the following generalized definition of the Restricted Isometry Property ($\mathrm{RIP}$).
\begin{definition} 
A matrix $\X\in\R^{n\times p}$ satisfies $\mathrm{RIP}(q, s, \delta)$ if and only if, for any $s$-sparse vector $\gamma$:
\begin{align*}
(1-\delta)\lVert\gamma\lVert_q\leq\lVert \X\gamma\lVert_q\leq(1+\delta)\lVert\gamma\lVert_q\,,
\end{align*}
\end{definition}
\noindent
Lemma \ref{Lemma UP} shows that an expander design $(1/\sqrt d)\X$ constructed from a $(2s,\varepsilon)$-unbalanced expander graph satisfies the $\mathrm{RIP}(1, s, 2\varepsilon)$ property, $\mathrm{RIP}_1$ for short. Conversely, any binary matrix $\X$ which satisfies the $\mathrm{RIP}_1$ property with proper parameters, and with each column having exactly $d$ ones, is an adjacency matrix of an unbalanced expander, see \cite{Ber}. Hence, expander design matrices are closely related to $\mathrm{RIP}_1$ property but they do not satisfy $\mathrm{RIP}_2$ property.

Restricted isometry property, for the case $q=2$, was introduced in \cite{MR2230846}. It was shown \cite{MR2230846, MR2382644} that if $\X$ satisfies this property, then the lasso and the Dantzig selector uncover a sparse approximation of the signal. Since then there has been a tremendous amount of work on $\mathrm{RIP}_2$ matrices. Unfortunately, expander designs cannot satisfy the $\mathrm{RIP}_2$ property, unless their number of rows is large \cite{chandar2008negative}. As a matter of fact, sparse binary matrices must have at least $n=\mathcal O(s^2)$ rows.
\subsection{Coherence property}

In $2007$, \cite{MR2543688} obtained an estimate in prediction for the lasso. They used a so-called {coherence property} following the work of \cite{MR2237332}. For any design matrix satisfying the {coherence property}, Theorem 1.2 in \cite{MR2543688} shows that, with high probability, it holds:
\begin{equation*}
\frac1n\norm{\X\beta^\star-\X\beta^{{\ell}} }_2^2\leq C'\,.\,\sigma^2
\,\frac{s\log p}n\,,
\end{equation*}
for a large set of $s$-sparse vectors $\beta^{\star}$, where $C'>0$ is some positive numerical constant. The
coherence $\mu$ is the maximum correlation between pairs of predictors:
\[
 \mu=\sup_{1\leq i< j\leq p} {\X_i^\top \X_j}
 \,.
\]
The coherence property \cite{MR2543688} is then $ \mu\leq A_0 (\log p) ^{-1}$, where $A_0$ is some positive constant. This property allows to deal with {random} design matrices. 

In the context of adjacency matrices, one can check that the coherence property is equivalent to the definition of a $(2,\varepsilon_0)$-unbalanced expander graph with an expansion constant $\varepsilon_0$ such that:
\[
\varepsilon_0=\frac\mu2\leq \frac{A_0}2 (\log p) ^{-1} \,,
\]
which is severely restrictive and cannot be used in the frame of expander codes.

\subsection{Restricted eigenvalue and compatibility condition}

The {restricted eigenvalue assumption}  \cite{MR2533469} and the compatibility condition \cite{MR2576316} consider the smallest eigenvalue with respect to a cone restriction. For the sake of simplicity, we present only the compatibility condition but the same analysis can be carried out for the restricted eigenvalue assumption.
\begin{definition}[$\mathrm{Compatibility}(s,c_0)$
]\label{def:Compatibility}
A matrix $\X\in\R^{n\times p}$ satisfies $\mathrm{Compatibility}(s,c_0)$ if and only if:
\begin{equation*}
 \phi(s,c_0) := \min_{\substack{\S\subseteq\{1,\dotsc,p\}\\\abs \S\leq
s}}\min_{\substack{\gamma\neq 0\\\lVert{\gamma_{\S^c}}\lVert_{1}\leq
c_0\lVert{\gamma_{\S}}\lVert_{1}}}\frac{\sqrt{\abs \S}\lVert{\X\gamma}
\lVert_{2}}{\lVert{\gamma_{\S}}\lVert_{1}}>0\,.
\end{equation*}
The constant $\phi(S,c_0)$ is called the compatibility constant or the $(S,c_0)$-restricted $\ell_1$-eigenvalue.
\end{definition}
\noindent 
One can established that \cite{MR2576316}, with high probability, if $\beta^{\star}$ is an $s$-sparse vector then:
\begin{equation}\label{eq:Compatibility}
\frac1{\sqrt n}\norm{\X\beta^\star-\X\beta^{{\ell}} }_2\leq C''\,.\,\frac{\sigma
}{\phi(s,3)}\,.\,\sqrt{\frac{s\log p}n}\,,
\end{equation}
where $C''>0$ is some positive constant {depending on $\phi(s,3)$ the $(s,3)$-restricted
$\ell_1$-eigenvalue}. Note that, invoking \eqref{eq:UDP18}, we derive the following proposition. 
\begin{proposition}\label{prop:REExpander}
Let $\Phi\in\{0,1\}^{n\times p}$ be the adjacency matrix of a $(2s,\varepsilon)$-unbalanced
expander with an expansion  constant $\varepsilon\leq1/12$ and left degree $d$. Set $\X=(1/{\sqrt d})\Phi$. If the quantities $1/\varepsilon$ and $d$ are smaller than $p$ then for all $c_0<4$, for all $S\subseteq\{1,\dotsc,p\}$ such that $\abs S\leq s$, and for all $\gamma\neq 0$ such that $\lVert{\gamma_{S^c}}\lVert_1\leq c_0\lVert{\gamma_{S}}\lVert_1$, it holds:
\begin{equation}\label{eq:BorneRE}
\frac{\sqrt{\abs \S}\lVert{\X\gamma}
\lVert_{2}}{\lVert{\gamma_{\S}}\lVert_{1}}\geq\frac{4-c_{0}}{6}\,.
\end{equation}
Hence, the $(s,c_0)$-restricted $\ell_1$-eigenvalue is lower bounded by the right hand side of \eqref{eq:BorneRE}.
\end{proposition}

\noindent
Thus, expander design matrices satisfy the compatibility condition with a constant $\phi(s,3)=1/6$. Note the same conclusion can be drawn for the restricted eigenvalue assumption.

\subsection{The Juditsky-Nemirovski condition}\label{Nemirovski}

In parallel to our work, \cite{JN10} gave a verifiable condition of performance of the lasso and the Dantzig selector. Although the matrices constructed from the expander graphs are not specifically studied in \cite{JN10}, they study uncertainty conditions similar to the one stated in \eqref{eq:UDP18}. More precisely, Section $5.3$ in \cite{JN10} says that a design $\X$ satisfies $\mathbf{H}_{s,1}(1/5)$ if and only if for all $\gamma\in\R^{p}$, for all $S\subseteq\{1,\dotsc,p\}$ such that $\abs S\leq s$,
\[
\lVert\gamma_S\lVert_1\leq\hat\lambda\,s\,\lVert \X\gamma\lVert_2+\frac15\lVert\gamma\lVert_1\,,
 \]
 for some constant $\hat\lambda>0$. 
Observe that \eqref{eq:UDP18} is a stronger requirement than $\mathbf{H}_{s,1}(1/5)$. Hence \eqref{eq:UDP18} implies $\mathbf{H}_{s,1}(1/5)$. Thus Lemma \ref{Lem:part1} shows that expander design matrices satisfy $\mathbf{H}_{s,1}(1/5)$.

\subsection{Universal distortion property}

In \cite{de2013remark}, one presents Universal Distortion Property. 

\begin{definition}[$\mathrm{UDP}(S_0,\kappa_0, \Delta )$]
Given $1\leq S_0\leq p$ and $0<\kappa_0<1/2$, we say that a matrix
$\X\in\R^{n\times p}$ satisfies the universal distortion condition of order
$S_0$,
magnitude $\kappa_0$ and parameter $\Delta$ {if and
only if} for all $ \gamma\in\R^p $, for all integers $ s\in\{1,\dotsc,S_0\}$, for all
subsets
$\S\subseteq\{1,\dotsc,p\}$ such that $\abs\S=s$, it holds:
\begin{equation*}
\norm{\gamma_\S}_{{1}}\leq\Delta\sqrt
s\,\norm{\X \gamma}_{{2}}+\kappa_0\norm{\gamma}_{1}.
\end{equation*}
\end{definition}
\noindent
Lemma \ref{Lem:part1} shows that if the quantities $1/\varepsilon$ and $d$ are smaller than $p$ then the normalized adjacency matrix of an $(s,\varepsilon)$-unbalanced expander with an expansion  constant $\varepsilon<1/2$ and left degree $d$ satisfies $\mathrm{UDP}(S_0,\kappa_0, \Delta )$ with $S_{0}=s$, $\Delta=1/(1-2\varepsilon)$, and $\kappa_0=2\varepsilon/(1-2\varepsilon)$.

\subsection{Comparison of the standard conditions}

Let $\X=(1/\sqrt d)\Phi$ where $\Phi$ is the adjacency matrix of an unbalanced expander graph with left degree $d$. One sees that $\X$ satisfies neither the $\mathrm{RIP}_2$ property nor the coherence property. However, the aforementioned subsections show that $\X$ satisfies Restricted Eigenvalue assumption, Compatibility Condition, $\mathbf{H}_{s,1}(1/5)$ and Universal Distortion Property with explicit constants. We have chosen this latter to derive our results but one would have get the same upper bound of the $\ell_{1}$-risk and $\ell_{2}$-prediction from the other conditions, see \cite{de2013remark} for further details.

\section{Deterministic construction of design matrices using Parvaresh-Vardy codes}\label{Deterministic Design}

A long standing issue in error-correcting code theory is to give deterministic and polynomial time constructions of expander codes \cite{MR2590822}. These constructions have already been used in the compressed sensing framework, see \cite{MR2582882} for instance. To the best of our knowledge, this paper is the first work that uses these constructions with the lasso and the Dantzig selector. The state-of-the-art constructions of expander codes use {Parvaresh-Vardy codes} \cite{Par}. More precisely, Guruswami, Umans, and Vadhan have recently proved the following theorem.

\begin{theorem}[\cite{MR2590822}]\label{Explicit Construction}
There exists a universal constant $\theta_0>0$ such that the following holds. For all $\alpha>0$ and for all $p,s,\varepsilon>0$, there exists a deterministic polynomial time construction of a $(2s,\varepsilon)$-unbalanced expander graph $G=(A,B,E)$ with $\abs A=p$, left degree:
\begin{equation*}
d\leq\big( (\theta_0\,{\log p}\log s)/\varepsilon\big)^{1+\frac1\alpha}\,,
\end{equation*}
and right side vertices $($of size $n=\abs B)$ such that:
\begin{equation}\label{right size}
 n\leq s^{1+\alpha}\big( (\theta_0\,{\log p}\log s)/\varepsilon\big)^{2+\frac2\alpha}\,.
\end{equation}
Moreover, $d$ is a power of $2$.
\end{theorem}

\noindent Note the size $n$ may depend on $p$ and other parameters.

\subsection{Probabilistic construction}
Using Chernoff bounds and Hoeffding's inequality, the following proposition can be shown.

\begin{proposition}[Theorem 4 in \cite{Has}]
\label{prop:Random}
Consider $\varepsilon>0$, $c>1$ and $p\geq 2s$. Then, with probability greater than $1-s\exp(-c\log(p))$, there exists an $(s,\varepsilon)$-unbalanced expander graph $G=(A,B,E)$ with $\abs A=p$, left degree $d$ such that $d\leq C_{1}(c,\varepsilon)\log(p)$ and number of right side vertices, namely $n=\abs B$, such that:
\begin{equation*}
n\leq C_{2}(c,\varepsilon)\,s\log(p)\,,
\end{equation*}
where $C_{1}(c,\varepsilon),C_{2}(c,\varepsilon)$ do not depend on $s$ but on $\varepsilon$.
\end{proposition}

\noindent Thus, with high probability, the normalized adjacency matrix of a random bipartite graph with a number of left side vertices $p$ and a number of right side vertices satisfying:
\begin{equation}\label{eq:BorneOptCS}
n\leq C\,s\log(p)\,,
\end{equation}
where $C>0$ is a universal constant, is an expander code design matrix. Observe that the bound \eqref{eq:BorneOptCS} is optimal, up to a subtractive $s\log(s)$ factor, see for instance Proposition 2.2.18 in \cite{chafai2011interactions}. Hence, random expander design matrices match the optimal bound \eqref{eq:BorneOptCS} on the number of observations.


\subsection{Deterministic construction}

Consider the following polynomial time construction \cite{MR2590822} of an expander code design:
\begin{enumerate}
 \item Choose $p$ the size of the signal $\beta^{\star}$, and $s$ the size of the sparse approximation,
  \item Set $\varepsilon=1/12$ the expansion constant,
  \item Set {$\alpha=1/\log(s)$} a tuning parameter,
  \item Construct an $(2s,\varepsilon)$-unbalanced expander graph $G$, using Theorem \ref{Explicit Construction}.
  \item Set $\X=(1/{\sqrt d})\Phi$, where $\Phi\in\{0,1\}^{n\times p}$ denotes the adjacency matrix of the graph $G$ and $d$ its left degree.
\end{enumerate}
\noindent Observe that the number of observation $n$ satisfies \eqref{right size}. Moreover, Proposition \ref{prop:REExpander} shows $\X$ satisfies the compatibility condition.
\begin{proposition}
With the same constant $\theta_{0}>0$ as in Theorem \ref{Explicit Construction}, the following holds. For all $s,p$ such that $8\leq s\leq p/2$, there exists a deterministic polynomial time construction of a design matrix $\X\in\R^{n\times p}$, where:
\begin{equation}\label{eq:BorneLassoMesures}
n\leq e\,s\,(12\,\theta_{0}\,\log p\log s )^{3\log(s)}\,,
\end{equation}
 with $(s,c_0)$-restricted $\ell_1$-eigenvalue $\phi(s,c_{0})$ larger than $\frac{4-c_{0}}{6}$.
 \end{proposition}
\begin{proof}
Invoke \eqref{right size} with $\alpha=1/\log(s)$ and $\varepsilon=1/12$:
\begin{align*}
n&\leq s^{1+1/\log(s)}\big(12\,\theta_0\,{\log p}\log s\big)^{2+2\log(s)}\,,\\
& \leq e\,s\big(12\,\theta_0\,{\log p}\log s\big)^{3\log(s)}\,,
\end{align*}
using $\log(s)\geq 2$ for $s\geq 8$. Using Proposition \ref{prop:REExpander}, we finish the proof.
\end{proof}

\noindent
Observe that the number of observations \eqref{eq:BorneLassoMesures} is almost optimal, see \eqref{eq:BorneOptCS} for the optimal bound in compressed sensing theory.

\subsection{Lasso case}
Using the aforementioned deterministic construction we derive the following results.
\begin{proposition}\label{prop:principalLassoDeterministe}
With the same constant $\theta_{0}>0$ as in Theorem \ref{Explicit Construction}, the following holds. For all $s,p$ such that $8\leq s\leq p/2$, there exists a deterministic polynomial time construction of a design matrix $\X\in\R^{n\times p}$ where:
\begin{equation*}
n\leq e\,s\,(12\,\theta_{0}\,\log p\log s )^{3\log(s)},
\end{equation*}
satisfying that for all $\lambda_{{\ell}}>10\,\sigma\sqrt{\log p}/3$, it holds:
\begin{equation*}
\norm{\beta^{{\ell}}-\beta^{\star}}_{1}\leq\frac{2}{\big(1-\frac{\lambda_0}{\lambda_{{\ell}}}
\big)-\frac{2}{5}
} \
\min_{\substack{\S\subseteq\{1,\dotsc,p\},\\ \abs{\S}=k,\ k\leq s.}}
\Big[\frac{36}{25}\,\lambda_{{\ell}}\,k+\norm{\beta^\star_{\S^c}}_{1}\Big],
\end{equation*}
where $\lambda_0:=2\,\sigma\sqrt{\log p}$, and:
\begin{equation*}
\norm{\X\beta^{{\ell}}-\X\beta^{\star}}_{{2}}\leq\min_{\substack{\S\subseteq\{1,\dotsc,p\},\\ \abs{\S}=k,\
k\leq s.}}
\Bigg[\frac{24}5\lambda_{{\ell}}\,\sqrt k+\frac{5\norm{\beta^\star_{\S^c}}_{1}}{6\sqrt k}\Bigg],
\end{equation*}
with a probability greater than $1-1/p{\sqrt{2\pi\log p}}$.
\end{proposition}
\begin{proof}
The proof follows from Theorem \ref{thm:principalLasso} and Theorem \ref{Explicit Construction}. 
\end{proof}

\subsection{Dantzig selector case}
\begin{proposition}\label{prop:principalDantzigDeterministe}
With the same constant $\theta_{0}>0$ as in Theorem \ref{Explicit Construction}, the following holds. For all $s,p$ such that $8\leq s\leq p/2$, there exists a deterministic polynomial time construction of a design matrix $\X\in\R^{n\times p}$ where
\[
n\leq e\,s\,(12\,\theta_{0}\,\log p\log s )^{3\log(s)},
\]
satisfying that for all $\lambda_{\mathrm d}>3\,\sigma\sqrt{\log p}$, it holds:
\begin{equation*}
\norm{\beta^{\mathrm d}-\beta^{\star}}_{1}\leq\frac{4}{\big(1-\frac{\lambda_0}{\lambda_{\mathrm d}}
\big)-\frac{1}{3}
} \
\min_{\substack{\S\subseteq\{1,\dotsc,p\},\\ \abs{\S}=k,\ k\leq s.}}
\Big[\frac{36}{25}\,\lambda_{\mathrm d}\,k+\norm{\beta^\star_{\S^c}}_{1}\Big],
\end{equation*}
where $\lambda_0:=2\,\sigma\sqrt{\log p}$, and:
\begin{equation*}
\norm{\X\beta^{\mathrm d}-\X\beta^{\star}}_{{2}}\leq\min_{\substack{\S\subseteq\{1,\dotsc,p\},\\ \abs{\S}=k,\
k\leq s.}}
\Bigg[\frac{24}5\lambda_{\mathrm d}\,\sqrt k+\frac{5\norm{\beta^\star_{\S^c}}_{1}}{6\sqrt k}\Bigg],
\end{equation*}
with a probability greater than $1-1/p{\sqrt{2\pi\log p}}$.
\end{proposition}
\begin{proof}
The proof follows from Theorem \ref{thm:PrincipalDantz} and Theorem \ref{Explicit Construction}. 
\end{proof}

\section{Conclusions}
We consider the design matrices derived form unbalanced expander graphs and show that we can use them to recover an $s$-sparse approximation of any signal of size $p$ using the lasso or the Dantzig selector. We show that one needs only $\mathcal O(s\log(p))$ measurements in the random case, and only $\mathcal O(s(\log(p)\log(s))^{3\log(s)})$ using a deterministic polynomial time construction.

\begin{figure}[!t]
\hspace*{-2.3cm}
\begin{tabular}{cc}
\includegraphics[height=5cm]{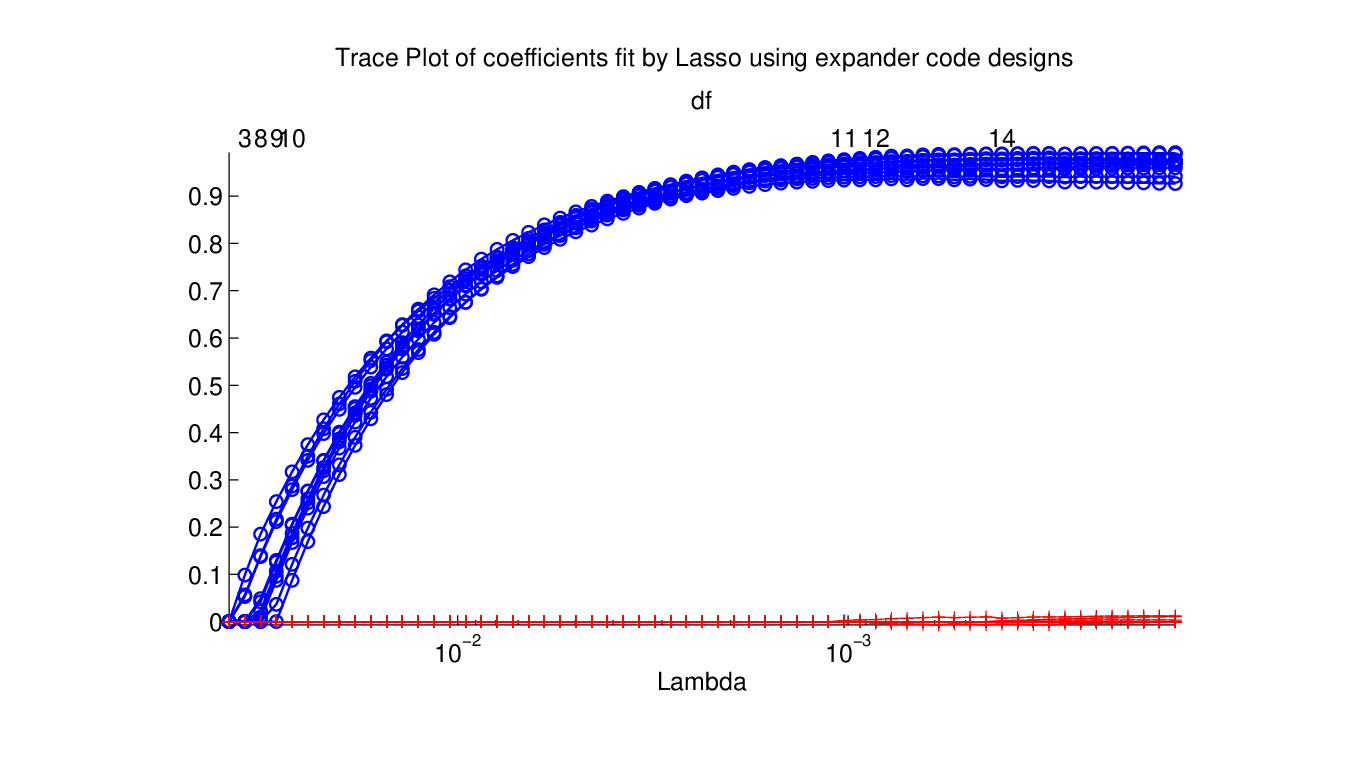}
& 
\hspace*{-1.4cm}\includegraphics[height=5cm]{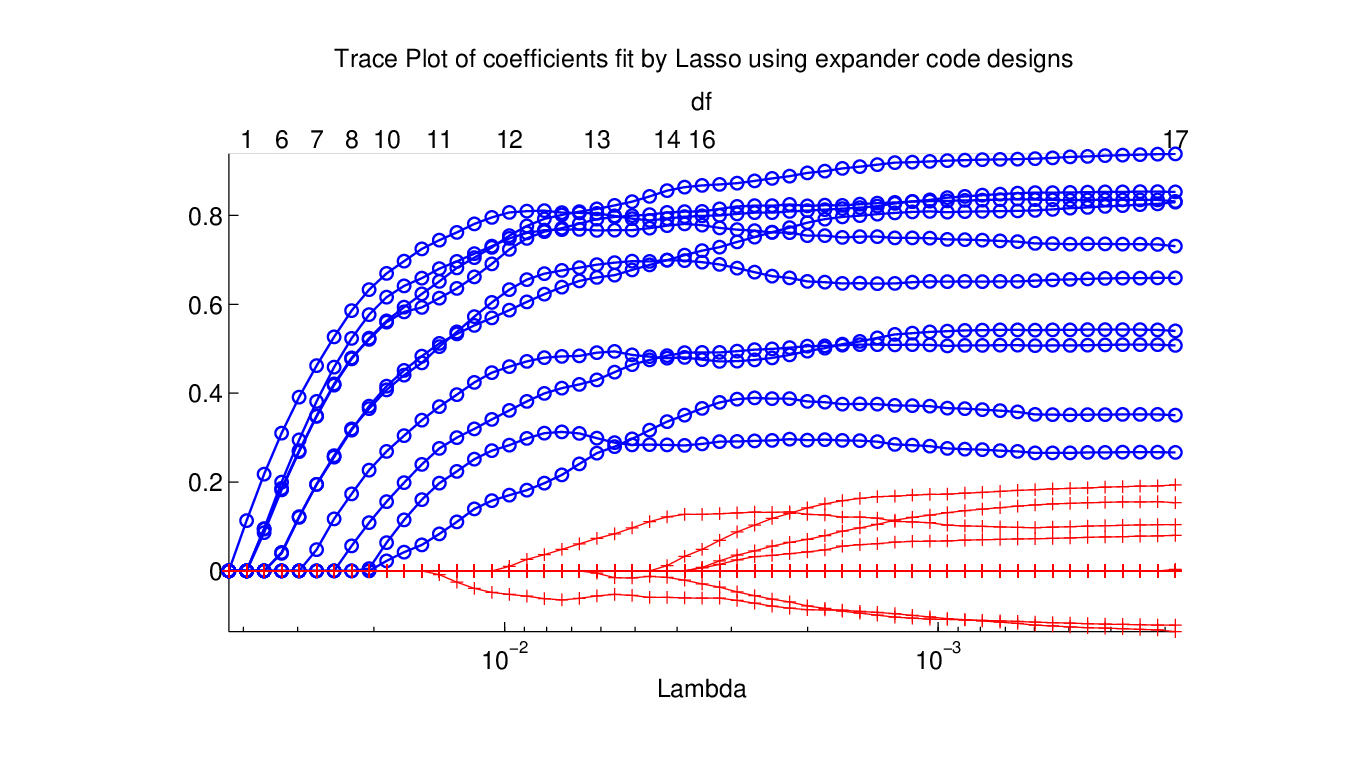}
\\
\includegraphics[height=5cm]{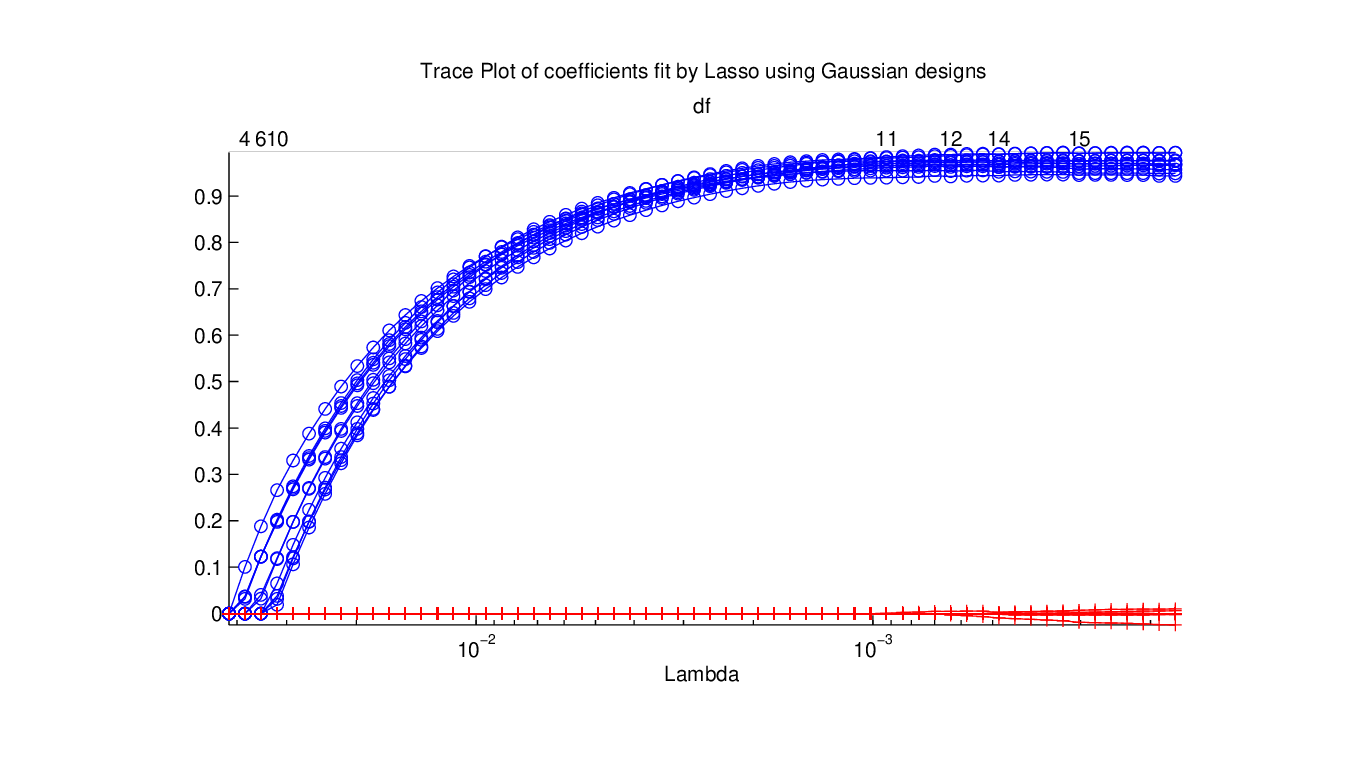} & 
\hspace*{-1.4cm}\includegraphics[height=5cm]{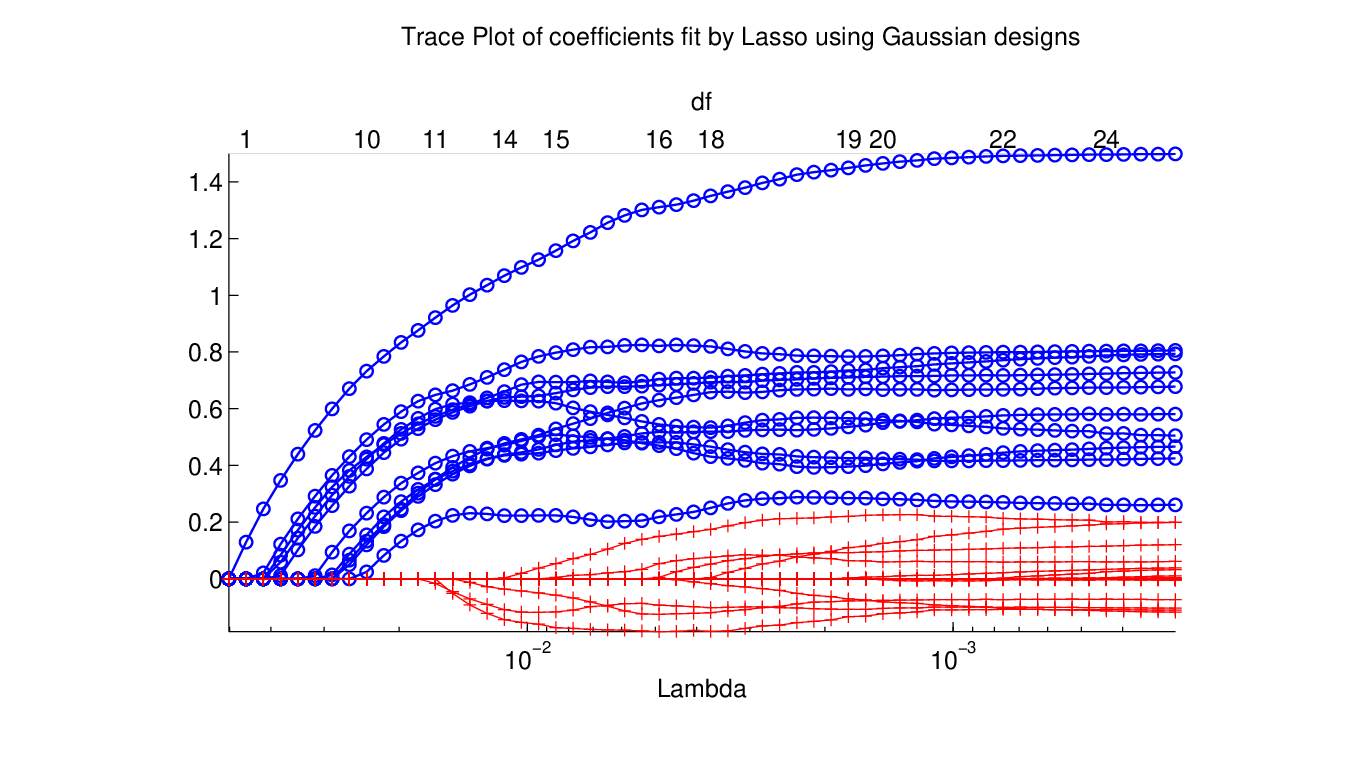}
\\
$\mathrm{SNR}=0\,\mathrm{dB}$
&
\hspace*{-1.4cm}
$\mathrm{SNR}=-6\,\mathrm{dB}$
\end{tabular}
\caption{Solution paths of Lasso using expander code designs (top panel) or Gaussian designs (bottom panel). The original target $\beta^{\star}$ has $10$ non-zero coefficients equal to $1$ (signal amplitude $\lVert\beta^{\star}\lVert_{\infty}$ is one) as in Figure \ref{Fig:fig1}. The blue circles shows the estimated coefficients at the true non-zero entries while the red crosses indicates the zero entries in the target. In these experiments we have chosen $p=5000$, $n=900$, $s=10$. In the left panel $\sigma=0.015$ while in the right panel $\sigma=0.2$. As expander code, we have drawn a uniform bi-partite graph with left degree $d=60$. The Gaussian matrix has i.i.d. $\mathcal N(0,1/n)$ entries. Observe lasso exhibits the same behavior with expander code designs or Gaussian designs. }\label{Fig:fig2}
\end{figure}

Moreover, we have run numerical experiments to compare random expander code design performances with Gaussian design performances. These latter are well-established designs in high-dimensional regression, see for instance \cite{chafai2011interactions,MR2300700,de2013remark}. Interestingly, expander code designs compares favorably and exhibit the same performances as Gaussian designs, see Figure \ref{Fig:fig2}. Simulations agrees with the theoretical guarantees of this paper. As a matter of fact, lasso using expander code designs enjoy the same (almost) optimal upper bound on its risk and prediction errors as lasso using Gaussian designs.

\vspace*{0.5cm}
\noindent\textbf{Acknowledgments} --- The author would like to thank anonymous referees for their fruitful remarks, their time and their patience.

 \bibliographystyle{alpha}
 \bibliography{ExpPred}
 
\end{document}